\documentclass[11pt]{amsart}
\usepackage[letterpaper, margin=1.in]{geometry}
\usepackage{graphicx}
\usepackage{amsmath,amssymb,amsthm,mathtools}
\usepackage[all]{xy}
\usepackage[utf8]{inputenc}
\usepackage[hidelinks]{hyperref}
\usepackage{enumitem}
\usepackage{xcolor}
\usepackage{tikz-cd}
\usepackage{tikz}
\usepackage{import}
\usepackage{marginnote,color}

\newtheorem{theorem}{Theorem}
\newtheorem{definition}{Definition}

\newtheorem{proposition}{Proposition}

\newtheorem{remark}{Remark}

\newtheorem{example}[theorem]{Example}

\newtheorem{conjecture}{Conjecture}
\newtheorem*{theorem*}{Theorem}
\newtheorem*{question*}{Question}
\newcommand{\cX}{\mathcal X}

\title{Bounding the number of graph refinements for Brill--Noether existence}
\date{\today}

\author{Karl Christ}
\address{Institute of algebraic geometry, Leibniz University Hannover, Welfengarten 1, 30167 Hannover, Germany
\newline
Current: Dipartimento di Matematica\\
	Università di Torino\\Via Carlo Alberto 10 \\10123 Turin\\  Italy }
\email{karl.christ@unito.it}

\author{Qixiao Ma}
\address{Institute of algebraic geometry, Leibniz University Hannover, Welfengarten 1, 30167 Hannover, Germany
\newline 
Current: Insitute of Mathematical Sciences, ShanghaiTech University, 201210, Shanghai, China 
}
\email{maqx1@shanghaitech.edu.cn}

\begin{document}

\begin{abstract}
Let $G$ be a finite graph of genus $g$. Let $d$ and $r$ be non-negative integers such that the Brill--Noether number is non-negative. Then the Brill--Noether existence conjecture due to Baker predicts the existence of a divisor of degree $d$ and rank at least $r$ on $G$. 

The conjecture is known to be true on the $k$-th homothetic refinement $G^{(k)}$ of $G$, for $k$ sufficiently large. Here we use results from classical Brill--Noether theory to give a uniform upper bound for $k$ in terms of $g,d,$ and $r$. We also discuss some algebro-geometric aspects of the conjecture and of recent counterexamples to a related conjecture  found in \cite{MR4397104}. 

\end{abstract}

\maketitle

\section{Introduction}
Let $g,d,$ and $r$ be non-negative integers such that the Brill--Noether number is non-negative,  $$\rho(g,d,r):=(r+1)(d-r)-g\cdot r\geq0.$$ 
Let $C$ be a smooth proper curve of genus $g$ defined over $\mathbb{C}$. The classical Brill--Noether existence theorem %\cite[V.(1.1)]{MR770932} 
asserts that there always exists a degree $d$ divisor $D$ on $C$, whose rank $r(D)=h^0(C,\mathcal{O}_C(D))-1$ is at least $r$. 

In \cite{BN}, Baker and Norine introduced a divisor theory on finite graphs, that mirrors large parts of the algebro-geometric behavior. In \cite{MR2448666}, Baker provided the first link between the algebro-geometric and combinatorial worlds via his specialization lemma, and since then this relationship has been extensively studied. Among the original conjectures in \cite{MR2448666}, Brill--Noether existence for finite graphs is the last one that remains open: 
\begin{conjecture}[\cite{MR2448666} Conjecture 3.9(1)]{\label{c1}}
Fix integers $g,d,r\geq0$, such that $\rho(g,d,r)=(r+1)(d-r)-g\cdot r \geq 0$. Then every graph of genus $g$ admits a divisor $D$ with $\mathrm{deg}(D)=d$ and $r(D)\geq r$.
\end{conjecture}

The conjecture is not known even for $r = 1$, despite considerable recent interest and many results for specific classes of graphs; see, for example, \cite{MR2895184}, \cite{MR3770859}, \cite{M22}, \cite{MR4397104} and \cite{MR4492530}. While the problem itself is purely combinatorial, the conjecture would follow from certain algebro-geometric statements, following an approach of Caporaso \cite[Section 6]{MR2895184}. We propose three versions of such an algebro-geometric statement, strictly increasing in strength, in Section~\ref{sec:further discussion}.

The existence of divisors predicted by Conjecture~\ref{c1} is known when we treat $G$ as a metric $\mathbb{Q}$-graph \cite[Theorem 3.12]{MR2448666}, or equivalently, the existence of such divisors is known on the $k$-th homothetic refinement $G^{(k)}$ for $k$ sufficiently large \cite[Theorem 1.3]{MR3092681}. The refinement $G^{(k)}$ of $G$, however, admits more divisors than $G$ and recent examples of \cite{MR4397104} show that they can behave quite differently (disproving a related conjecture, \cite[Conjecture 3.14]{MR2448666}). These examples were an important motivation for our work, and we show that they are not an artifact of the combinatorial side of the theory; that is, they are obtained from algebraic geometry via specialization. See Section~\ref{subsec:example}.

Our main result is to use known results from algebraic geometry to establish the following upper bound on $k$, which significantly improves a previous bound obtained in \cite{MR4397104} with combinatorial methods. See Section~\ref{subsec:comparison} for a comparison.
\begin{theorem}\label{main}
 Let $G$ be any graph of genus $g$. Let $d$ and $r$ be non-negative integers so that $\rho:=(r+1)(d-r)-g\cdot r\geq0$. Then there always exists a divisor $D\in \mathrm{Div}^d(G^{(k)})$  with $r_{G^{(k)}}(D)\geq r$
  for some \begin{equation*}\label{eq:bound}
      k \leq g!\cdot \prod_{i=0}^r \frac{i!}{(g-d+r+i)!}.
  \end{equation*}
\end{theorem}

The upper bound in Theorem~\ref{main} is equal to the degree of the Brill--Noether locus with
respect to the theta divisor on a general algebraic curve. The core of the
argument is that subdividing the graph enough to ensure that
this enumerative problem has a rational solution guarantees the existence of the desired divisor.

\medskip

\noindent 
{\bf Acknowledgements.} We would like to thank the anonymous referees for their helpful comments.

\section{Preliminaries}
\subsection{Basics on graphs}\label{subsec:basics} 
We collect some basic concepts and notations from \cite{MR2448666}:
\begin{enumerate}
 \item \textbf{Graph:} By a graph $G$, we always mean a finite weightless connected multigraph without loop edges. We denote the set of vertices and edges of $G$ by $V(G)$ and $E(G)$, respectively. The genus $g$ of the graph is given by $$g(G)=1-|V(G)|+|E(G)|.$$
 \item \textbf{Divisors and degree:} Let $\mathrm{Div}(G)$ be the free abelian group on $V(G)$. An element $D\in\mathrm{Div}(G)$ is called a divisor on $G$. Explicitly, we can write $$D=\sum_{v\in V(G)}a_v\cdot (v),\ a_v\in\mathbb{Z}.$$ The degree of $D$ is defined as $\mathrm{deg}(D)=\sum_{v\in V(G)}a_v$. For a fixed integer $d$, we denote the set of divisors of degree $d$ on $G$ by $\mathrm{Div}^d(G)$.
% \item \textbf{Effective divisors:} 
We call a divisor $D\in\mathrm{Div}(G)$ effective, if $D=\sum_{v\in V(G)} a_v\cdot(v)$ with $a_v\geq0$ for all $v\in V(G)$. 
 %Let $\mathrm{Div}_+(G)\subset\mathrm{Div}(G)$ be the subset of effective divisors on $G$. Let $$\mathrm{Div}_+^d(G)=\mathrm{Div}_+(G)\cap\mathrm{Div}^d(G)$$ be the set of all effective divisors of degree $k$ on $G$.
 \item \textbf{Linear equivalence:} To a $\mathbb{Z}$-valued function $f$ on $V(G)$, we may associate a divisor on $G$ whose coefficient at $v \in V(G)$ is $$\sum_{e=vw\in E(G)}(f(v)-f(w)).$$
 Divisors of this form have degree $0$ and are called principal divisors. Two divisors $D,D'\in \mathrm{Div}(G)$ are linearly equivalent, written $D\sim D'$, if $D-D'$ is a principal divisor. The complete linear system associated to $D$ is given by $$|D|=\{E\in\mathrm{Div}(G)\colon E\sim D \text{ and } E \text{ is effective } \}.$$
 \item \textbf{Rank of divisors:} Given a divisor $D$ on $G$, its rank $r_G(D)$ is $-1$ if $|D|=\emptyset$, and otherwise given as $$r_G(D)=\mathrm{max}\{k\in\mathbb{Z}\colon |D-E| \neq \emptyset,\ \text{ for all effective divisors } E\in\mathrm{Div}^k(G)\}.$$
\end{enumerate}

\subsection{Refinements of graphs}Let $G$ be a graph. A refinement of $G$ is a graph $G'$ obtained by inserting a finite set of vertices in the interior of the edges of $G$. The natural inclusion of vertices $$V(G)\subset V(G')$$ induces an injective group homomorphism
$$\iota_{G,G'}\colon\mathrm{Div}(G)\to\mathrm{Div}(G').$$ The map $\iota_{G,G'}$ does not necessarily preserve the rank of the divisors, i.e., in general $r_G(D)\neq r_{G'}(\iota_{G,G'}(D))$ (this is for example the case for the divisor on $G'$ in Figure~\ref{fig: tricycle} as discussed in Section~\ref{subsec:example}). The rank however is preserved, if we insert the same number of vertices in each edge:

%by \cite[Theorem 1.5]{MR2793609} and \cite[Corollary 3.4]{MR3092681}.

%\footnote{If one allows loop edges for $G$, this is no longer the case, see \cite[Example 5.4]{MR2895184}.} 
%However, the situation is much better if we consider homothetic refinements of the graph.

\begin{definition} For an integer $k\geq 1$, let $G^{(k)}$ be the refinement of $G$ obtained by subdividing every edge of $G$ into $k$ edges by inserting $k-1$ vertices. We call the refinement  $G^{(k)}$ of $G$ the $k$-th homothetic refinement of $G$. %We call $n$ the order of the homothetic refinement $G\subset G^{(n)}$.
\end{definition}

\begin{proposition}[\cite{MR3092681} Corollary 3.4]For every $D\in\mathrm{Div}(G)$ and every $k\geq0$, we have $$r_G(D)=r_{G^{(k)}}(\iota_{G,G^{(k)}}(D)).$$ 
\end{proposition}
Therefore, if the Brill--Noether existence conjecture holds for a graph $G$, it holds for any homothetic refinement $G^{(k)}$ of $G$.

\section{Proof of Theorem \ref{main}}

\subsection{Regular one-parameter smoothing}
Let $G$ be a graph of genus $g$ and $X_0$ a nodal curve over $\mathbb{C}$ with dual graph $G$. By this we mean in particular, that every irreducible component of $X_0$ is isomorphic to $\mathbb{P}^1_{\mathbb{C}}$ since $G$ is weightless.

As usual, we denote by $\mathbb C[[t]]$ the ring of formal power series with coefficients in $\mathbb C$ and by $\mathbb C((t))$ the field of Laurent series, its field of fractions.

By \cite[Theorem B.2]{MR2448666}, we may spread out the curve $X_0/\mathbb{C}$ to a family of nodal curves over $\mathrm{Spec}(\mathbb{C}[[t]])$: $$\pi\colon\mathcal{X}\to T=\mathrm{Spec}(\mathbb{C}[[t]]),$$ 
\begin{enumerate}
 \item with special fiber $X_0$, 
 \item with generic fiber $\cX_\eta$ a smooth proper curve of genus $g$, and
 \item such that the total space $\mathcal{X}$ is regular.
\end{enumerate}
We will call $\mathcal X$ a regular one-parameter smoothing of $X_0$. We denote by $\eta\in T$ the generic point and by $\mathcal{X}_\eta$ the generic fiber.
 
\subsection{The $k$-th root fibration}\label{base-change} Let $k\geq1$ be an integer. Let $\pi\colon \mathcal{X}\to T$ be a regular one-parameter smoothing of $X_0$.
Let $r_k\colon T=\mathrm{Spec}(\mathbb{C}[[t]])\to T$ be the morphism induced by $t\mapsto t^k$. Let $\rho_k\colon\mathcal{X}_1=\mathcal{X}\times_{T,r_k}T\to T$ be the base change of $\pi$ along $r_k$. Let $\nu\colon\mathcal{X}_2\to\mathcal{X}_1$ be the normalization, and let $d\colon\mathcal{X}'\to\mathcal{X}_2$ be the minimal desingularization. We call the composition $$\pi^{(k)}:=
\rho_k\circ\nu\circ d\colon \mathcal{X}'\xrightarrow{d} \mathcal{X}_2\xrightarrow{\nu}\mathcal{X}_1\xrightarrow{\rho_k}T$$ the $k$-th root fibration of $\pi\colon\mathcal{X}\to T$, see \cite[III.9]{MR749574}. The total space $\mathcal{X}'$ is regular by our construction. The generic fiber of $\pi^{(k)}$, denoted by $
\mathcal{X}'_\eta$, is a smooth proper curve of genus $g$. Let $X_0'$ be the special fiber of $\pi^{(k)}$.
\begin{proposition} \label{prop:dual graph}
The curve $X'_0$ is a nodal curve with smooth rational components, and has dual graph $G^{(k)}$.
\end{proposition}
\begin{proof}
We keep track of the special fiber through each step in the construction of the $k$-th root fibration. Clearly 
the special fiber remains unchanged by base change. Since $X_0$ is reduced, it remains unchanged by normalization. The effect of the minimal desingularization 
follows from a local calculation: Each node of $X_0$ is locally defined by $xy-t=0$. After base change and normalization, the local equation becomes $xy-t^k=0$, which is an $A_{k-1}$ singularity. The minimal resolution of an $A_{k-1}$
singularity is given by iterated blow-ups, which inserts a chain of $k - 1$ smooth rational curves at each node of the special fiber. See, for example, \cite[Section 10.3, Corollary 3.25]{Liu} for further details.
\end{proof}

\subsection{The Picard scheme} We denote by $\pi'\colon \mathrm{Pic}^d_{\mathcal{X}/T}\to T$ the relative Picard scheme representing families of line bundles of total degree $d$ on $\pi\colon\mathcal{X}\to T$. Since we work over a strictly henselian base, the set $\mathrm{Pic}_{\cX/T}(T)$ of $T$-sections of the Picard scheme represents line bundles on $\mathcal{X}$. 

Let us denote the generic fiber of $\pi'$ by $\mathrm{Pic}^d(\mathcal{X}_\eta)$. The special fiber is a countable union of algebraic tori, indexed by the multi-degrees of line bundles on $X_0$; the multi-degree $\underline{\mathrm{deg}}(L_0)$ of a line bundle $L_0$ on $X_0$ is a tuple of integers, one for each irreducible component of $X_0$, given by the degree of the restriction of $L_0$ to that component. We may view such a multi-degree as an element of $\mathrm{Div}(G)$, where $G$ is the dual graph of $X_0$.

\subsection{The specialization lemma} Our proof relies on 
Baker's specialization lemma. As before, we let $\pi\colon\mathcal{X}\to T=\mathrm{Spec}(\mathbb{C}[[t]])$ be a regular one-parameter smoothing of a nodal curve $X_0$ that has dual graph $G$. Consider the specialization map $$\tau\colon\mathrm{Pic}_{\mathcal{X}/T}^d(T)\to\mathrm{Div}^d(G),\ \ \mathcal L\mapsto\underline{\mathrm{deg}}(\mathcal L|_{X_0})$$ which sends a relative line bundle to the multi-degree of its restriction to the central fiber. Let $\eta\in T$ denote the generic point and $\mathcal{X}_\eta$ the generic fiber:

\begin{proposition}[Specialization lemma, \cite{MR2448666} Lemma 2.8]\label{mixed}
   In the above setting, we have $$r(\cX_\eta,\mathcal L|_{\mathcal{X}_\eta})\leq r_G(\tau(\mathcal L))$$
   for any $\mathcal L \in \mathrm{Pic}_{\mathcal{X}/T}^d(T)$.
\end{proposition}

\subsection{The Brill--Noether locus}Let $r$ and $d$ be non-negative integers that satisfy $\rho=(r+1)(d-r)-g\cdot r\geq0$. Over an algebraically closed field, by the classical Brill--Noether theorem \cite[V.(1.1)]{MR770932}, we know that for a smooth curve $C$ of genus $g$, the Brill--Noether locus $$W_d^r(C)=\{L\in\mathrm{Pic}^d(C)\colon h^0(C,L)\geq r+1\}\subset\mathrm{Pic}^d(C)$$ is nonempty. More generally, the locus can be naturally equipped with a scheme structure, which then is non-empty for smooth curves over an arbitrary base field. %In particular, for $r=1$ and $d=\lfloor\frac{g+3}{2}\rfloor$, we always have $\rho=2(\lfloor\frac{g+3}{2}\rfloor-1)-2g\geq0$.

\medskip

For any nodal curve $X_0$ with dual graph $G$, we consider the Brill--Noether scheme $W_d^r(\cX_\eta)\subset\mathrm{Pic}^d(\cX_{\eta})$ on the generic fiber of a regular one-parameter smoothing $\cX$ of $X_0$. 
\begin{proposition}\label{prop3}
 If there exists a closed point $Q$ in $W_d^r(\cX_\eta)$ of degree $k$, then taking the $k$-th root fibration $\pi^{(k)}\colon\mathcal{X}'\to T$, the generic fiber $W_d^r(\cX'_\eta)$ admits a $\mathbb{C}((t))$-rational point. 
\end{proposition}
\begin{proof}
Let $Q\in W_d^r(\cX_\eta)$ be a closed point of degree $k$. It corresponds to a degree-$k$ extension of $\mathbb{C}((t))$. By \cite[Theorem 1.94]{MR2289519}, we know that $Q\cong\mathrm{Spec}(\mathbb{C}((t^{1/k})))$ and therefore splits after we take the $k$-th root base change $r_k\colon T\to T$. Since the formation of Brill--Noether schemes commutes with base change, this yields a $\mathbb{C}((t))$-rational point of $W_d^r(\cX_\eta')$.
\end{proof}

\subsection{Proof of Theorem \ref{main}}
By assumption we have  $\rho:=(r+1)(d-r)-g\cdot r\geq0.$  Thus by the classical Brill--Noether theorem, we know that for a general curve $C$, the Brill--Noether scheme $W_r^d(C)$ has dimension exactly $\rho$.

Let $X_0$ be a nodal curve with smooth rational components and dual graph $G$. 
Let $\pi\colon\mathcal{X}\to T=\mathrm{Spec}(\mathbb{C}[[t]])$ be a regular one-parameter smoothing of $X_0$. Let $\cX_\eta$ be the generic fiber of $\pi$. By the proof of \cite[Theorem B.2]{MR2448666}, we may further ask $\cX_\eta$ to be Brill--Noether general. 

By Hensel's lemma, any $\mathbb{C}$-point in the smooth locus of $X_0$ lifts to a $\mathbb{C}[[t]]$-point of $\mathcal{X}$, and therefore the set $\cX_\eta(\mathbb{C}((t)))$ of $\mathbb{C}((t))$-points is Zariski dense in $\mathcal{X}_\eta$. Let $P$ be any fixed $\mathbb{C}((t))$-point in $\mathcal{X}_\eta$. Then, by twisting with $\mathcal{O}^{\otimes(d-g)}_{\mathcal{X}_\eta}(P)$, we may identify $\mathrm{Pic}^g(\mathcal{X}_\eta)$ with $\mathrm{Pic}^d(\mathcal{X}_\eta)$ for any $d$.

Note that $\mathrm{Pic}^g(\cX_\eta)$ is naturally dominated by the $g$-th self product of $\cX_\eta$ via the Abel-Jacobi map. So the set of $\mathbb{C}((t))$-points is Zariski dense in $\mathrm{Pic}^g(\cX_\eta)$, and therefore the set of $\mathbb{C}((t))$-points is also Zariski dense in $\mathrm{Pic}^d(\cX_\eta)$ for any $d$.

Let $\Theta=W_{g-1}^0(\cX_\eta)\subset\mathrm{Pic}^{g-1}(\cX_\eta)$ be the theta divisor. For any $L\in\mathrm{Pic}^{d+1-g}(\cX_\eta)$, let $t_L\Theta\subset\mathrm{Pic}^d(\cX_\eta)$ be the translate of $\Theta$ by $L$. 

Note that for any irreducible subscheme $Z\subset\mathrm{Pic}^d(\cX_\eta)$, the open subscheme $$U_{Z,\Theta}:=\{L\in\mathrm{Pic}^{d+1-g}(\cX_\eta)|Z\nsubseteq t_L\Theta\}\subset\mathrm{Pic}^{d+1-g}(\cX_\eta)$$ is non-empty and hence Zariski dense in $\mathrm{Pic}^{d+1-g}(\cX_\eta)$. Indeed, when $L$ runs through $\mathrm{Pic}^{d+1-g}(\cX_\eta)$, the translates $t_L^{-1}Z$ cover $\mathrm{Pic}^{g-1}(\cX_\eta)$. Therefore $t_L^{-1} Z \nsubseteq \Theta$ or equivalently $Z \nsubseteq t_L \Theta$ for general $L \in \mathrm{Pic}^{d+1-g}(\cX_\eta)$. %geometrically, the inverse translates of $Z$ by $\mathrm{Pic}^1(\cX_\eta)$ cover $\mathrm{Pic}^{g-1}(\cX_\eta)$, so some of them are not contained in $\Theta$, and intersect the ample divisor $\Theta$ properly. 

For an equi-dimensional subscheme $Z'=\cup_{i=1}^s Z_i\subseteq\mathrm{Pic}^d(\cX_\eta)$, let $U_{Z',\Theta}:=\cap_{i=1}^s U_{Z_i,\Theta}$. Then the intersection $Z'\cap t_L\Theta$ is an equi-dimensional subscheme of dimension $\dim Z'-1$ for any $L$ in the dense open subscheme $U_{Z',\Theta}\subseteq\mathrm{Pic}^{d + 1 - g}(\cX_\eta)$.

We fix a set of points $L_1,\cdots,L_\rho\in \mathrm{Pic}^{d+1-g}(\cX_\eta)$, that give a set of translates $t_{L_1}\Theta,\cdots,t_{L_\rho}\Theta$ of $\Theta$, which are all divisors in $\mathrm{Pic}^d(\mathcal{X}_\eta)$. By \cite[Theorem V.1.3]{MR770932}, we have for the intersection number of effective cycles on $\mathrm{Pic}^d(\cX_\eta)$: $$W_d^r(\cX_\eta)\cdot (t_{L_1}\Theta)\cdots(t_{L_\rho}\Theta)=g!\cdot\prod_{i=0}^r\frac{i!}{(g-d+r+i)!}.$$Let us view the left hand side as iterated intersections of $W_d^r(\cX_\eta)$ with translates of $\Theta$.
By a suitable choice of the $L_i$'s, we can ask at each step, that the intersection is dimensionally transverse. Therefore, the intersection product is represented by an effective zero cycle, and $W_d^r(\cX_\eta)$ admits a closed point of degree $$k \leq g!\cdot\prod_{i=0}^r\frac{i!}{(g-d+r+i)!}.$$

Let us take the $k$-th root fibration $\pi^{(k)}\colon\mathcal{X}'\to T$. By Proposition~\ref{prop:dual graph}, the dual graph of the central fiber of $\cX'$ is $G^{(k)}$. By Proposition~\ref{prop3}, the generic fiber $W_d^r(\cX_\eta')$ admits a $\mathbb{C}((t))$-rational point. Since $\cX'_\eta(\mathbb{C}((t)))\neq\emptyset$, the rational point is represented by a degree-$d$ line bundle $\mathcal{L}_\eta$ on $\cX'_\eta$ of rank $r(\cX'_\eta,\mathcal{L}_{\eta})$ at least $r$. Since $\mathcal{X}'$ is regular, $\mathcal{L}_\eta$ extends to some line bundle $\mathcal{L}$ on $\mathcal{X}'$. By Proposition~\ref{mixed}, the divisor $\tau(\mathcal{L})\in\mathrm{Div}^d(G^{(k)})$ has rank $r_{G^{(k)}}(\tau(\mathcal{L}))\geq r(\cX'_\eta,\mathcal{L}_\eta)\geq r$. This finishes the proof of Theorem~\ref{main}.

\subsection{Comparison to previous bound} \label{subsec:comparison}

In \cite[Remark 3.7]{MR4397104} a different upper bound for $k$ is given by \begin{equation} \label{eq:previous bound}
    k \leq (m+n^r d)!\cdot d^{m+n^r d},
\end{equation} where $n=|V(G)|$, $m=|E(G)|$.
The bound in Theorem~\ref{main} improves this bound significantly. 

Indeed, the bound in Theorem~\ref{main} depends only on $g, r$ and $d$. If we fix the value of $g$, the minimal value of $n$ is $2$ since $G$ is weightless and without loops; the minimal value of $m$ is $g + 1$, since each cycle contains at least two edges, and any two cycles need to differ by at least one edge (both values are realized for the graph with $2$ vertices with $g  + 1$ edges between them). Hence the minimal value for fixed $g$ of \eqref{eq:previous bound} is \[ (g + 1 + 2^r d)! \cdot d^{g + 1 + 2^r d}.\]

We claim that the bound in Theorem~\ref{main} is better already in this case. Indeed, we can estimate the bound in Theorem~\ref{main}, as follows:
\begin{eqnarray*}\label{eq:comparison}
 g! \cdot \prod_{i=0}^r \frac{i!}{(g-d+r+i)!} < g! (r!)^r < g! (d!)^r < g! d^{dr}.
\end{eqnarray*}
Now $g ! < (g + 1 + 2^r d)!$ and $d^{dr} < d^{g + 1 + 2^r d}$, which gives the claim. 

\section{Further discussion} \label{sec:further discussion}
\subsection{Sufficient algebro-geometric claims}\label{sufficiency} In \cite[Theorem 6.3]{MR2895184} a proof of Conjecture~\ref{c1} is given, that however contains a gap as pointed out in \cite[Footnote 5, p. 379]{MR3702316}. The approach of \cite[Theorem 6.3]{MR2895184} is to reduce Conjecture~\ref{c1} to known algebro-geometric statements via Baker's specialization lemma, along similar lines we used in the proof of Theorem~\ref{main}.  

Here are three related algebro-geometric statements, that would suffice to complete the argument (though none of them is necessary for Brill--Noether existence on finite graphs). Namely, Conjecture~\ref{c1} would follow, if one could choose for every graph $G$ a regular one-parameter smoothing $\cX \to T = \mathrm{Spec}(\mathbb C[[t]])$ of a curve $X_0$ with dual graph $G$ such that 
\begin{enumerate}
 \item the special fiber of the closure of the Brill--Noether scheme $W_{d}^r(\mathcal{X}_\eta)$ in $\mathrm{Pic}^d_{\cX/T}$ is not empty;
 \item the Brill--Noether scheme $W_d^r(\cX_\eta)$ admits $\mathbb C((t))$-rational points; or
 \item the special fiber of the closure of the Brill--Noether locus $W_{d}^r(\mathcal{X}_\eta)$ in some compactification of $\mathrm{Pic}^d_{\cX/T}$ has a reduced component. 
\end{enumerate}

These assertions are in increasing strength. That is, (3) implies (2) since if the special fiber of the closure of $W_d^r(\mathcal{X}_\eta)$ contains a reduced component, it contains smooth points that can be lifted to sections of $W_d^r(\cX_\eta)$ by Hensel's lemma. And (2) implies (1), since every line bundle on $\cX_\eta$ extends to a line bundle on $\cX$. Assertion (1) in turn implies Brill--Noether existence for graphs, via the argument proposed in \cite[Theorem 6.3]{MR2895184}. 

We conclude our discussion with two examples, that illustrate the above three conditions. 

By a slight abuse of notation, we call a line bundle of degree $d$ and rank at least $r$ on a curve $X$ a $g_d^r$. Similarly, we call a complete linear system of degree $d$ and rank at least $r$ on a graph $G$ a $\mathfrak g_d^r$.

\begin{remark}
    A potentially different approach to deduce Brill--Noether existence for finite graphs from an algebro-geometric statement involves the algebraic rank of Caporaso \cite{C13}, which gives a lower bound on the rank of a divisor on a graph $G$ by \cite{CLM} (see also \cite{BCM}). It is determined by the rank of line bundles on nodal curves with dual graph $G$. Thus it avoids all questions of rational points on smoothings discussed here; this of course comes at the expense of considerably more difficulties to bound the rank of line bundles on nodal curves than on smooth ones (see, e.g., \cite{C23} and \cite{C24} and the references there).  
\end{remark}

\subsection{Dependence on the choice of $X_0$ and on its smoothing}

Our first example illustrates the aforementioned conditions in the simplest nontrivial case: $g_3^1$'s on genus $4$ curves. Recall that by \cite[IV. Example 5.2.2]{MR0463157}, a non-hyperelliptic smooth genus $4$ curve $C$ is the complete intersection of a unique quadric surface $Q$  and a cubic surface $T$ in $\mathbb{P}^3$ via its canonical embedding. 

Let $\pi\colon\mathcal{U}\to P=|\mathcal{O}_{\mathbb{P}^3}(2)|\times |\mathcal{O}_{\mathbb{P}^3}(3)|$ be the universal family of $(2,3)$-complete intersections in $\mathbb{P}^3$. Let $S\subset P$ be the open subscheme over which $\pi$ is a family of stable curves, and the quadric $Q$ has at worst a node. Let $R\subset S$ be the locus where $Q$ is singular. Let $f\colon\mathcal{F}\to S$ be the universal Fano scheme of lines on $Q$. The morphsim $f$ factors through its scheme of components $g\colon \mathcal{W}\to S$, which is a double covering of $S$ ramified over $R$.

Let $p\colon \mathcal{X}\to S$ be the restriction of $\pi$ to $S$. Let $\mathrm{Pic}_{\mathcal{X}/S}$ be the relative Picard scheme. Restriction of Cartier divisors induces a morphism $\mathcal{W}\to \mathrm{Pic}^3_{\mathcal{X}/S}$. Note that the classifying morphism $S\to\overline{\mathcal{M}}_4$ is dominant, that $\rho(4,3,1)=2\cdot2-1\cdot4=0$ and that a general $W_3^1$ consists of $4!\frac{0!}{2!}\frac{1!}{3!}=2$ points. Therefore, $\mathcal{W}$ coincides with the universal Brill--Noether locus $\mathcal{W}_3^1$ over the generic point of $S$.

Since $g$ is finite and $S$ is irreducible, we know that $\mathcal{W}$ is an irreducible closed subscheme in $\mathrm{Pic}^3_{\mathcal{X}/S}$. 
Therefore, for any separated relative compactification $i\colon \mathrm{Pic}^3_{\mathcal{X}/S}\to \overline{\mathrm{Pic}}^3_{\mathcal{X}/S}$, the closure of the Brill--Noether scheme $\mathcal{W}_3^1$ coincides with the image of $\mathcal{W}$, and does not meet the boundary $\overline{\mathrm{Pic}}^3_{\mathcal{X}/S}\backslash i(\mathrm{Pic}^3_{\mathcal{X}/S})$. 
It follows, that condition (1) is always satisfied in this setting.

The following example shows how the claims of the list in Section \ref{sufficiency} depend on the choice of a curve $X_0$ with dual graph $G$ and its smoothing: 

\begin{example}
Let $G$ be the graph with $3$ vertices and two edges between each two of the vertices. Let $\mathcal{M}_G\subset\overline{\mathcal{M}}_4$ be the substack of curves with dual graph $G$. Cross ratios on the components give a finite chart of $\mathcal{M}_G$, hence $\mathcal{M}_G$ is irreducible of dimension $3$. Counting parameters of moduli, we know that a general curve with dual graph $G$ is the intersection of a smooth quadric $Q_0\subset \mathbb{P}^3$ with a union of three general planes $T=H_1\cup H_2\cup H_3$. The two rulings on $Q_0$ give two distinct $g_3^1$'s on $X_0=Q_0\cap T$ and for an arbitrary one-parameter smoothing, all three conditions in Section~\ref{sufficiency} are satisfied.
 
Suppose now, that $Q_0$ is a nodal quadric. We already observed that condition (1) is still satisfied; which of the remaining conditions are satisfied now depends on the choice of the one-parameter smoothing. Let $B$ be the local ring of $S$ at $[X_0]$, and let $b\in B$ be the defining equation of $R\subset S$.
The morphism $g\colon \mathcal{W}\to S$ is of the form $\mathrm{Spec}B[T]/(T^2-b)\to \mathrm{Spec}B$ locally over $[X_0]\in S$.

Let $\mathcal{X}\to D=\mathrm{Spec}(\mathbb{C}[[t]])$ be a one-parameter smoothing of $X_0$ induced by $f\colon D\to \mathrm{Spec}(B)$ with $f^*\colon B\to\mathbb{C}[[t]]$. Let $\cX_\eta$ be the generic fiber. Then $W_3^1(\cX_\eta)=\mathcal{W}_\eta$ has coordinate ring $\mathbb{C}((t))[T]/(T^2-f^*b)$.
\begin{itemize}[leftmargin=*]
\item If $f$ is transversal to $V(b)$, then $f^{-1}(V(b))$ is reduced, $f^*b$ is not a square and $T^2-f^*b$ is irreducible. Hence $\mathcal{W}_\eta$
admits no $\mathbb{C}((t))$-points, and condition (2) (hence also (3)) fails. 
\item If $f$ is simply tangent to $V(b)$, then $f^*b$ is a square in $\mathbb{C}((t))$ and $\mathcal{W}_\eta$ splits. But the special fiber $\mathcal{W}|_{[X_0]}\cong\mathrm{Spec}(\mathbb{C}[[T]]/(T^2))$ is non-reduced. So condition (3) fails for $X_0$, while condition (2) still holds. 
\end{itemize}
\end{example}

\subsection{The gonality of a graph may change after refinement}\label{subsec:example}

Recent examples of \cite{MR4397104} show, that the divisor theory on $G$ and $G^{(k)}$ can behave quite differently. The simplest of their examples is depicted on the left and middle in Figure~\ref{fig: tricycle}; the graph $G$ does not admit a divisor of rank $1$ and degree $5$, a $\mathfrak g_5^1$, whereas $G^{(2)}$ does \cite[Proposition 4.2 and Theorem 4.8]{MR4397104}. This is not a counterexample to Conjecture~\ref{c1}, as the Brill--Noether number in this case is $\rho = 2(5 - 1) - 9 \cdot 1 = -1.$ In this final section, we will discuss some algebro-geometric aspects of this example.
\tikzset{every picture/.style={line width=0.75pt}} %set default line width to 0.75pt
\begin{figure}[ht]
\begin{tikzpicture}[x=0.6pt,y=0.6pt,yscale=-0.7,xscale=0.7]
\import{./}{fig1.tex}
\end{tikzpicture}
\caption{On the left, a graph $G$ that does not admit a $\mathfrak g_5^1$. In the middle, the refinement $G^{(2)}$, together with a $\mathfrak g_5^1$. On the right, a (non-homothetic) refinement of $G$ and a $\mathfrak g_5^1$ supported on the vertices of $G$.}
\label{fig: tricycle}
\end{figure}

\subsubsection{Liftability of the $\mathfrak g_5^1$ on $G^{(2)}$.} First of all, we note that the  $\mathfrak g_5^1$ on $G^{(2)}$ is not an artifact of the combinatorial side of the theory, in the sense that it is the specialization of a $g_5^1$ on a smooth curve (since the specialization lemma only gives an inequality of ranks, it could happen that the $\mathfrak g_5^1$ on $G^{(2)}$ does not have an algebro-geometric counterpart). To this end, recall that the limits of rank $1$ line bundles can be explicitly described via the theory of admissible covers of Harris and Mumford \cite[\S 4]{HM}, to which we refer for the necessary definitions. 

The dual graph of such an admissible cover $X_0 \to Y_0$ is depicted in Figure~\ref{fig2} (it has no ramification over the nodes and otherwise only simple ramification points). On $G^{(2)}$ it gives the linear system of the  $\mathfrak g^1_5$ in Figure~\ref{fig: tricycle}.
Also depicted is a metric structure on the graphs that gives a tropical admissible cover $\Gamma \to \Sigma$, which encodes some information about possible smoothings (see \cite{ABBRI} and \cite{CMR}).

\tikzset{every picture/.style={line width=0.75pt}}
\begin{figure}[ht]
\begin{tikzpicture}[x=0.7pt,y=0.7pt,yscale=-0.8,xscale=0.8]
\import{./}{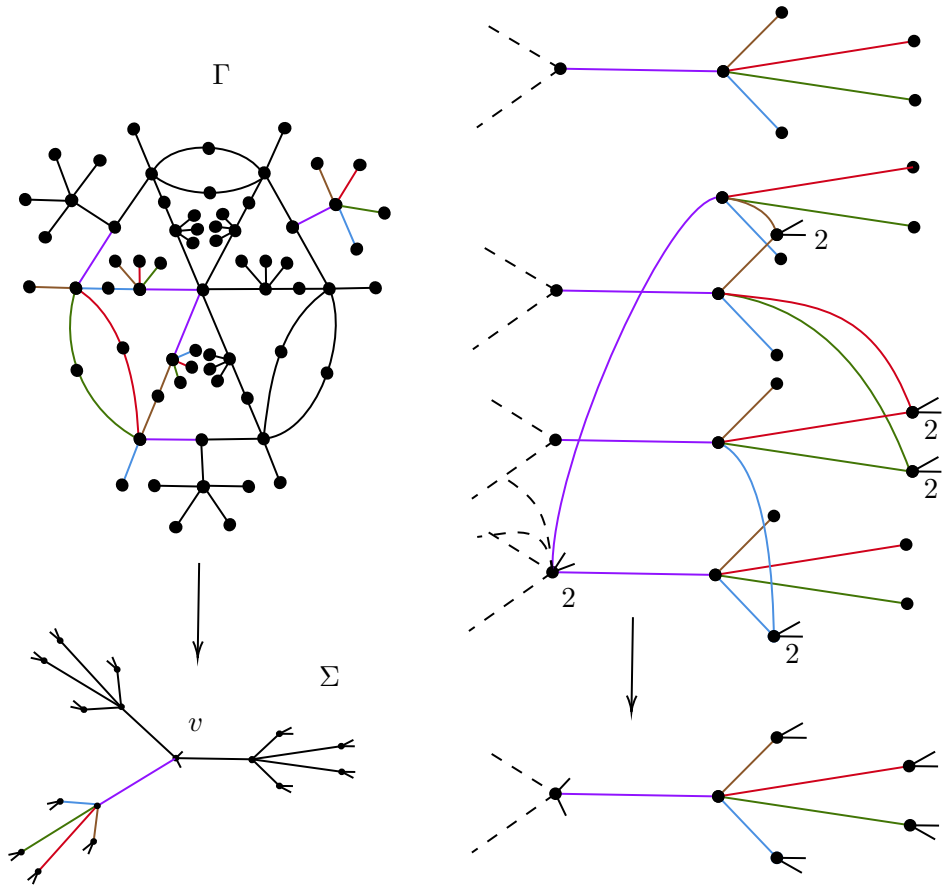}
\end{tikzpicture}
\caption{On the left, $\Gamma$ and $\Sigma$. On the right, the map $\Gamma \to \Sigma$ zoomed in at a subgraph of $\Gamma$; the map is defined analogously on the other parts of $\Gamma$. Numbers indicate the degree of the map on a vertex, if it is different from $1$; half-edges indicate branch or ramification points. The map has stretching factor $1$ along all edges and edges in the same fiber have the same colour. The blue and brown edges have half the length of the other edges.}
\label{fig2}
\end{figure}

Let us recall, how this admissible cover gives rise to a $g_5^1$ on a smooth curve that specializes to the $\mathfrak g_5^1$ on $G^{(2)}$. Let $p \in Y_0$ be a smooth point lying on the irreducible component of $Y_0$ corresponding to the vertex $v$. The underlying graph of $\Gamma$ depicted in Figure~\ref{fig2} can be obtained from $G^{(2)}$ by successively adding either a vertex in an edge or a vertex together with one adjacent edge. Let $\pi \colon X_0 \to X_0^{(2)}$ denote the map, which contracts the components of $X_0$ corresponding to these added vertices. The curve $ X_0^{(2)}$ hence has dual graph $G^{(2)}$. 

The admissible cover $f \colon X_0 \to Y_0$ is the limit of a $g_5^1$ on some one-parameter smoothing $\cX$ of $X_0$. Since $f^* \mathcal O_{Y_0}(p)$ has degree $0$ on all components contracted by $\pi$, $\pi_* f^* \mathcal O_{Y_0}(p)$ is a line bundle on $ X_0^{(2)}$. By construction, its multi-degree is the $\mathfrak g_5^1$ on $G^{(2)}$ depicted in Figure~\ref{fig: tricycle}. Now $\pi$ extends to a map $\cX \to \cX^{(2)}$ that is an isomorphism on the generic fiber and $\pi$ on the special fiber. It follows that also $\pi_* f^* \mathcal O_{Y_0}(p)$ is a limit of a $g_5^1$ on $\cX^{(2)}$. 

In summary, we obtain a line bundle in $\mathrm{Pic}^5( X_0^{(2)})$, that is the limit of a $g_5^1$ on a smooth curve, and whose multi-degree is the given $\mathfrak g_5^1$ on $G^{(2)}$.
Since the stabilization of the underlying graph of $\Gamma$ depicted in Figure~\ref{fig2} is $G$, it follows, in particular, that the closure of the locus of smooth curves of genus $9$ admitting a $g_5^1$ in $\overline {\mathcal M}_9$ intersects the boundary stratum $\mathcal M_G$ parametrizing curves with dual graph $G$.

\subsubsection{(Non)existence of limits of $g_5^1$'s as line bundles} Next, we consider other specializations of $g_5^1$'s to $\mathfrak g_5^1$'s on refinements of $G$. On any regular one parameter smoothing of a curve with dual graph $G$, no limit (as a line bundle) of a $g_5^1$ can exist, since $G$ does not admit a $\mathfrak g_5^1$. More generally, since the rank of a divisor on $G$ is preserved under homothetic refinement, we have the following: 

Suppose $\cX$ is a (not necessarily regular) one-parameter smoothing of a curve $X_0$ with dual graph $G$. Assume furthermore, that locally at each node, $\cX$ is given by $xy - t^k = 0$ \emph{for the same $k$ at each node}. Denote by $\pi \colon \cX' \to \cX$ the minimal desingularization obtained by blowing up $k - 1$ times each node of $X_0$ in the total space $\cX$. It is a regular one-parameter smoothing of a nodal curve $X'_0$ with dual graph $G^{(k)}$. Now if the generic fiber $\cX_\eta$ admits a $g_5^1$, $\mathcal L_\eta$, then there is no extension of $\mathcal L_\eta$ to $\cX$ that restricts to a line bundle on the central fiber $X_0$. 

Indeed, suppose to the contrary that such an extension $\mathcal L$ does exist. Then the pullback $\pi^* \mathcal L$ gives a line bundle that is a $g_5^1$ on the generic fiber, and whose multidegree on $G^{(k)}$ is supported only on vertices of $G$. By the specialization lemma, this multidegree would give a $\mathfrak g_5^1$ on $G^{(k)}$ and hence on $G$, which gives the contradiction. 

This is not the case, if we consider one-parameter smoothings corresponding to refinements of $G$ that are not homothetic. Indeed, consider the $\mathfrak g_5^1$ on the refinement $G'$ of $G$ depicted on the right in Figure~\ref{fig: tricycle}, which is supported on vertices of $G$. One can construct an admissible cover that realizes this divisor as we did for the $\mathfrak g_5^1$ on $G^{(2)}$ in Figure~\ref{fig2}. As above, the admissible cover gives a one-parameter smoothing $\cX''$ whose central fiber has dual graph a refinement of $G'$, together with a $g_5^1$ on the generic fiber that specializes to the $\mathfrak g_5^1$ on $G'$. This implies, that we can choose an extension of the $g_5^1$ on $\cX''_\eta$, whose central fiber is a line bundle whose multidegree is supported on vertices of $G$. Then the stabilization of $\cX''$ gives a (non-regular) one-parameter smoothing of a curve with dual graph $G$, on which a limit of a $g_5^1$ (as a line bundle) does exist.  

\bibliographystyle{alpha}

\end{document}